\documentclass[12pt,twoside]{amsart}
\usepackage{amssymb,amscd,amsxtra,calc,mathrsfs}
\usepackage{amsmath}
\usepackage{xypic}

\title[Hyperplane arrangements]{K-stability of log Fano hyperplane arrangements}
\author{Kento Fujita} 
\date{\today}
\address{Research Institute for Mathematical Sciences, Kyoto University, Kyoto 606-8502, Japan}
\email{fujita@kurims.kyoto-u.ac.jp}

\newcommand{\pr}{\mathbb{P}}

\newcommand{\Z}{\mathbb{Z}}
\newcommand{\Q}{\mathbb{Q}}
\newcommand{\R}{\mathbb{R}}
\newcommand{\C}{\mathbb{C}}

\newcommand{\LL}{\mathbb{L}}

\newcommand{\A}{\mathbb{A}}
\newcommand{\G}{\mathbb{G}}

\newcommand{\Supp}{\operatorname{Supp}}

\newcommand{\Spec}{\operatorname{Spec}}

\newcommand{\DIV}{\operatorname{div}}

\newcommand{\Aut}{\operatorname{Aut}}
\newcommand{\Proj}{\operatorname{Proj}}
\newcommand{\codim}{\operatorname{codim}}

\newcommand{\lct}{\operatorname{lct}}
\newcommand{\DF}{\operatorname{DF}}

\newcommand{\ord}{\operatorname{ord}}

\newcommand{\vol}{\operatorname{vol}}

\newcommand{\NA}{\operatorname{NA}}

\newcommand{\PGL}{\operatorname{PGL}}
\newcommand{\SL}{\operatorname{SL}}
\newcommand{\diag}{\operatorname{diag}}

\newcommand{\mult}{\operatorname{mult}}

\newcommand{\sP}{\mathcal{P}}
\newcommand{\sO}{\mathcal{O}}

\newcommand{\sX}{\mathcal{X}}
\newcommand{\sY}{\mathcal{Y}}
\newcommand{\sL}{\mathcal{L}}

\newcommand{\sF}{\mathcal{F}}

\newcommand{\sW}{\mathcal{W}}

\newcommand{\sslash}{\mathbin{/\mkern-6mu/}}

\newtheorem{thm}{Theorem}[section]
\newtheorem{lemma}[thm]{Lemma}
\newtheorem{proposition}[thm]{Proposition}
\newtheorem{corollary}[thm]{Corollary}

\theoremstyle{definition}
\newtheorem{definition}[thm]{Definition}
\newtheorem{remark}[thm]{Remark}

\newtheorem{example}[thm]{Example}
\newtheorem*{ack}{Acknowledgments}

\begin{document}

\maketitle 

\begin{abstract}
In this article, we completely determine which log Fano hyperplane arrangements are 
uniformly K-stable, K-stable, K-polystable, K-semistable or not. 
\end{abstract}

\setcounter{tocdepth}{1}
\tableofcontents

\section{Introduction}\label{intro_section}

Let $(X, \Delta)$ be a \emph{log Fano pair}, that is, $X$ is a normal projective variety 
over the complex number field $\C$, $\Delta$ is an effective 
$\Q$-divisor on $X$ such that the pair $(X, \Delta)$ is klt and 
$-(K_X+\Delta)$ is ample. (For the minimal model program, we refer the readers to 
\cite{KoMo}. For example, ``klt" stands for ``kawamata log terminal", and 
``lc" stands for ``log canonical".) 
We are interested in the question whether $(X, \Delta)$ is 
\emph{uniformly K-stable}, \emph{K-stable}, \emph{K-polystable}, 
\emph{K-semistable} or not (see 
\cite{tian, don, stoppa, B, CDS1, CDS2, CDS3, tian2} and \cite{sz, dervan, BHJ, BBJ, fjt} 
for example). 
We recall the definition of those stability conditions in \S \ref{K_section}. 
However, in general, it is hard to determine those stability conditions 
for given log Fano pairs. 

In this article, we focus on log Fano hyperplane arrangements. 

\begin{definition}[{see \cite{mustata, teitler}}]\label{ha_dfn}
Let $(X, \Gamma)$ be an $n$-dimensional \emph{hyperplane arrangement}, that is, 
$X$ is equal to $\pr^n$ and $\Gamma$ is of the form $\sum_{i=1}^md_iH_i$, where 
$d_i\in\Q_{>0}$ and $H_1,\dots,H_m$ are mutually distinct hyperplanes. 
\begin{enumerate}
\renewcommand{\theenumi}{\arabic{enumi}}
\renewcommand{\labelenumi}{(\theenumi)}
\item\label{ha_dfn1}
$(X, \Gamma)$ is said to be a \emph{log Fano hyperplane arrangement} if 
$(X, \Gamma)$ is a log Fano pair. 
\item\label{ha_dfn2}
$(X, \Gamma)$ is said to be a \emph{log Calabi-Yau hyperplane arrangement} if 
$K_X+\Gamma\sim_\Q 0$. Moreover, if $(X, \Gamma)$ is lc (resp., klt), then 
$(X, \Gamma)$ is said to be an \emph{lc Calabi-Yau hyperplane arrangement} 
(resp., a \emph{klt Calabi-Yau hyperplane arrangement}). 
We sometimes say  in \S \ref{P_section} that 
the pair $(\pr^0=\Spec \C, \emptyset)$ is a zero-dimensional 
klt Calabi-Yau hyperplane arrangement. 
\item\label{ha_dfn3}
We set 
\begin{eqnarray*}
\LL(X, \Gamma)&:=&\left\{\bigcap_{i\in I}H_i, \Big|\, I\subset\{1,\dots,m\}\right\}, \\
\LL'(X, \Gamma)&:=&\LL(X, \Gamma)\setminus\{X, \emptyset\}.
\end{eqnarray*}
(We set $\LL(X, \Gamma):=\{X$, $\emptyset\}$ if $n=0$.)
\item\label{ha_dfn4}
For any $W\in\LL(X, \Gamma)$, we set 
\begin{eqnarray*}
d_{(X, \Gamma)}(W)&:=&\sum_{W\subset H_i}d_i, \\
c^X(W) & := & \codim_X W.
\end{eqnarray*}
(We set $c^X(\emptyset):=n+1$.) Moreover, we set 
\[
d_{(X, \Gamma)}:=d_{(X, \Gamma)}(\emptyset)=\sum_{i=1}^md_i.
\]
(We set $d_{(X, \Gamma)}(X):=0$ and $d_{(X, \Gamma)}(\emptyset):=1$ if $n=0$.)
\end{enumerate}
\end{definition}

One-dimensional log Fano pairs are 
always log Fano hyperplane arrangements since $X$ must be isomorphic to $\pr^1$. 
We know that uniform K-stability, 
K-semistability of one-dimensional log Fano pairs is well-understood 
(see \cite[Theorem 3]{liR} and \cite[Example 6.6]{fjt}) and the condition is classically 
well-known (see \cite{troyanov, mcowen, CL, CY, LT} for example). In particular, 
the condition is numerical; we can understand uniform K-stability and K-semistability 
of one-dimensional log Fano hyperplane arrangements by only 
looking at the coefficients of its boundaries. 

\begin{thm}[{see \cite[Theorem 3]{liR} and \cite[Example 6.6]{fjt}}]\label{one_thm}
Let $(X, \Delta)=(\pr^1, \sum_{i=1}^m d_ip_i)$ be a one-dimensional log Fano pair, 
where $p_i$ are distinct points and $d_i\in(0,1)\cap\Q$. Then $(X, \Delta)$ is uniformly 
K-stable $($resp., K-semistable$)$ if and only if $\sum_{j\neq i}d_j> d_i$ $($resp., 
$\sum_{j\neq i}d_j\geq d_i)$ holds for any $1\leq i\leq m$. 
\end{thm}

We generalize Theorem \ref{one_thm} 
for log Fano hyperplane arrangements of any dimension. 
In particular, we will show that the conditions for 
uniform K-stability, K-stability, K-polystability and K-semistability of log Fano 
hyperplane arrangements depend only on those combinatorial informations. 

We recall the notion of log canonical thresholds. 

\begin{definition}[Log canonical thresholds]\label{lct_dfn}
Let $(V, \Delta_V)$ be a klt pair with $\Delta_V$ effective 
$\Q$-divisor and let $M$ be an effective $\Q$-Cartier $\Q$-divisor on $V$. 
We set 
\[
\lct(V, \Delta_V; M):=\max\{a\in\Q_{>0}\,|\, (V, \Delta_V+aM)\text{ is lc}\}.
\]
\end{definition}

If $V=\pr^n$, $\Delta_V=0$ and $(V, M)$ is a hyperplane arrangement, then 
the value $\lct(V, \Delta_V; M)$ is well-understood. 

\begin{thm}[{\cite{mustata, teitler}}]\label{mustata_thm}
Let $(X, \Gamma)$ be a hyperplane arrangement with $\Gamma\neq 0$. Then 
\[
\lct(X, 0; \Gamma)=\min_{W\in\LL'(X, \Gamma)}
\left\{\frac{c^X(W)}{d_{(X, \Gamma)}(W)}\right\}.
\]
\end{thm}

Using Theorem \ref{mustata_thm}, a valuative criterion for K-stability 
introduced in \cite{li, fjt, fjt17} and so on, we get the following theorem. 

\begin{thm}[Main result]\label{mainthm}
Let $(X, \Delta)$ be an $n$-dimensional 
log Fano hyperplane arrangement with $\Delta\neq 0$. Set 
\[
\Gamma:=\frac{n+1}{d_{(X, \Delta)}}\Delta.
\]
\begin{enumerate}
\renewcommand{\theenumi}{\arabic{enumi}}
\renewcommand{\labelenumi}{(\theenumi)}
\item\label{mainthm1}
The following are equivalent: 
\begin{enumerate}
\renewcommand{\theenumii}{\roman{enumii}}
\renewcommand{\labelenumii}{(\theenumii)}
\item\label{mainthm11}
$(X, \Delta)$ is K-semistable, 
\item\label{mainthm12}
$(X, \Gamma)$ is an lc Calabi-Yau hyperplane arrangement, 
\item\label{mainthm13}
for any $W\in\LL'(X, \Delta)$, we have 
\[
\frac{c^X(W)}{d_{(X, \Delta)}(W)}\geq \frac{n+1}{d_{(X, \Delta)}}.
\]
\end{enumerate}
\item\label{mainthm2}
The following are equivalent: 
\begin{enumerate}
\renewcommand{\theenumii}{\roman{enumii}}
\renewcommand{\labelenumii}{(\theenumii)}
\item\label{mainthm21}
$(X, \Delta)$ is uniformly K-stable, 
\item\label{mainthm22}
$(X, \Delta)$ is K-stable, 
\item\label{mainthm23}
$(X, \Gamma)$ is a klt Calabi-Yau hyperplane arrangement, 
\item\label{mainthm24}
for any $W\in\LL'(X, \Delta)$, we have 
\[
\frac{c^X(W)}{d_{(X, \Delta)}(W)}> \frac{n+1}{d_{(X, \Delta)}}.
\]
\end{enumerate}
\item\label{mainthm3}
The following are equivalent: 
\begin{enumerate}
\renewcommand{\theenumii}{\roman{enumii}}
\renewcommand{\labelenumii}{(\theenumii)}
\item\label{mainthm31}
$(X, \Delta)$ is K-polystable, 
\item\label{mainthm32}
$(X, \Gamma)$ is a log Calabi-Yau hyperplane arrangement of class $\sP$ 
$($for the definition of class $\sP$, see \S \ref{P_section}$)$.
\end{enumerate}
\end{enumerate}
\end{thm}

As an immediate consequence of Theorem \ref{mainthm}, we get the following 
corollary: 

\begin{corollary}\label{maincor}
Let $(X, \Delta)=(\pr^n, \sum_{i=1}^md_iH_i)$ be an $n$-dimensional 
log Fano hyperplane arrangement, where $H_i$ are distinct hyperplanes and $d_i\in(0,1)\cap\Q$. 
Assume that $\Delta\neq 0$ and $\Supp\Delta$ is simple normal crossing. 
\begin{enumerate}
\renewcommand{\theenumi}{\arabic{enumi}}
\renewcommand{\labelenumi}{(\theenumi)}
\item\label{maincor1}
The following are equivalent: 
\begin{enumerate}
\renewcommand{\theenumii}{\roman{enumii}}
\renewcommand{\labelenumii}{(\theenumii)}
\item\label{maincor11}
$(X, \Delta)$ is K-semistable $($resp., uniformly K-stable$)$, 
\item\label{maincor12}
\[
k\sum_{i=1}^md_i\geq (n+1)\sum_{j=1}^kd_{i_j}\quad
\left(\text{resp., }k\sum_{i=1}^md_i>(n+1)\sum_{j=1}^kd_{i_j}\right)
\]
holds for any $1\leq k\leq n$ and for any $1\leq i_1<\cdots<i_k\leq m$. 
\end{enumerate}
\item\label{maincor1}
The following are equivalent: 
\begin{enumerate}
\renewcommand{\theenumii}{\roman{enumii}}
\renewcommand{\labelenumii}{(\theenumii)}
\item\label{maincor21}
$(X, \Delta)$ is K-polystable, 
\item\label{maincor22}
$(X, \Delta)$ is uniformly K-stable, or 
\[
m=n+1\quad\text{and}\quad d_1=\cdots=d_{n+1}.
\]
\end{enumerate}
\end{enumerate}
\end{corollary}

\begin{remark}\label{mainthm_rmk}
\begin{enumerate}
\renewcommand{\theenumi}{\arabic{enumi}}
\renewcommand{\labelenumi}{(\theenumi)}
\item\label{mainthm_rmk1}
The condition for class $\sP$ is purely combinatorial. See \S \ref{P_section}. 
\item\label{mainthm_rmk2}
If $\Delta=0$, then $(X, \Delta)=(\pr^n, 0)$ is K-polystable but not K-stable 
by \cite[Theorem 1.1]{B}. We remark 
that K-semistability of $(X, \Delta)=(\pr^n, 0)$ is classical 
(see \cite{kempf, don} for example) and is proved purely algebraically 
(see \cite{li, blum, BJ} for example). 
The proof of Theorem \ref{mainthm} is purely algebraic, and 
we only use \emph{K-semistability} of $(\pr^n, 0)$ for the proof of Theorem \ref{mainthm}. 
\item\label{mainthm_rmk3}
The equivalence between the condition \eqref{mainthm12} and the condition 
\eqref{mainthm13} in Theorem \ref{mainthm}, 
and the equivalence between the condition \eqref{mainthm23} and the condition 
\eqref{mainthm24} in Theorem \ref{mainthm} follow immediately from Theorem 
\ref{mustata_thm}. 
\end{enumerate}
\end{remark}

The strategy for the proof of Theorem \ref{mainthm} is the following. 
We firstly consider a valuative criterion for K-stability, K-polystability of 
log Fano pairs. More precisely, we will prove the following: 

\begin{thm}[{=Theorem \ref{stps_thm}}]\label{stps_intro_thm}
Let $(X, \Delta)$ be a log Fano pair. 
\begin{enumerate}
\renewcommand{\theenumi}{\arabic{enumi}}
\renewcommand{\labelenumi}{(\theenumi)}
\item\label{stps_intro_thm1}
The following are equivalent: 
\begin{enumerate}
\renewcommand{\theenumii}{\roman{enumii}}
\renewcommand{\labelenumii}{(\theenumii)}
\item\label{stps_intro_thm11}
$(X, \Delta)$ is K-stable, 
\item\label{stps_thm12}
$\hat{\beta}_{(X, \Delta)}(F)>0$ holds for any dreamy prime divisor $F$ over $(X, \Delta)$ 
$($see Definition \ref{vol_dfn}$)$.
\end{enumerate}
\item\label{stps_intro_thm2}
The following are equivalent: 
\begin{enumerate}
\renewcommand{\theenumii}{\roman{enumii}}
\renewcommand{\labelenumii}{(\theenumii)}
\item\label{stps_intro_thm21}
$(X, \Delta)$ is K-polystable, 
\item\label{stps_intro_thm22}
$\hat{\beta}_{(X, \Delta)}(F)\geq 0$ holds for any dreamy prime divisor $F$ over 
$(X, \Delta)$, and equality holds only if $F$ is a product-type prime divisor over 
$(X, \Delta)$ $($see Definition \ref{prod_div_dfn}$)$.
\end{enumerate}
\end{enumerate}
\end{thm}

Secondly, we will see the following easy sufficient condition for log Fano pairs being 
uniformly K-stable or K-semistable.

\begin{proposition}[{see Proposition \ref{technical_prop}}]\label{technical_intro_prop}
Let $(X, \Delta)$ be a K-semistable log Fano pair, set $L:=-(K_X+\Delta)$, 
let $r$ be a rational number with $r\in(0,1)$, 
and let $B$ be an effective $\Q$-divisor on $X$ with $B\sim_\Q rL$. 
Set $a:=\lct(X, \Delta; B)$. Assume that $a\geq r^{-1}$. 
Then $(X, \Delta+B)$ is a K-semistable log Fano pair. 
Moreover, if $(X, \Delta)$ is uniformly K-stable or $a>r^{-1}$, then 
$(X, \Delta+B)$ is a uniformly K-stable log Fano pair. 
\end{proposition}

We can easily prove Theorem \ref{mainthm} \eqref{mainthm1} and \eqref{mainthm2} 
from the above observation plus some easy calculations (see \S \ref{blowup_section}). 
The proof of 
Theorem \ref{mainthm} \eqref{mainthm3} is relatively difficult. We must know several properties of 
product-type prime divisors over $(X, \Delta)$. See \S \ref{dreamy_section} and \S 
\ref{P_section}.

This article is organized as follows. In \S \ref{K_section}, we recall the definition for 
K-stability of log Fano pairs. Moreover, we recall a valuative criterion for 
uniform K-stability, K-semistability of log Fano pairs introduced in \cite{li, fjt, fjt17}. 
In \S \ref{dreamy_section}, we introduce a valuative criterion for \emph{K-stability}, 
\emph{K-polystability} of log Fano pairs. In \S \ref{blowup_section}, we see 
a necessary condition for K-stability of log Fano hyperplane arrangements via 
blowup along linear subspaces. In \S \ref{P_section}, we define the notion of 
log Calabi-Yau hyperplane arrangements of class $\sP$. 
Moreover, we give a combinatorial characterization of class $\sP$. 
In \S \ref{unif_section}, we give a sufficient condition for log Fano pairs being uniformly 
K-stable or K-semistable in terms of log canonical thresholds. 
In \S \ref{poly_section}, we characterize K-polystable log Fano hyperplane arrangements. 
In \S \ref{GIT_section}, we see a relationship to GIT stability. 
In \S \ref{alpha_section}, we give examples of $n$-dimensional non-K-polystable 
log Fano hyperplane arrangements $(X, \Delta)$ such that the alpha invariants 
$\alpha(X, \Delta)$ of $(X, \Delta)$ are equal to $n/(n+1)$. 

In this article, we do not distinguish $\Q$-Cartier $\Q$-divisors and $\Q$-line bundles. 
For varieties $V_1$ and $V_2$, the first (resp., the second) projection is denoted by 
$p_1\colon V_1\times V_2\to V_1$ (resp., $p_2\colon V_1\times V_2\to V_2$). For 
a projective morphism of varieties $\sX\to C$ with $\sX$ normal and 
$C$ a smooth curve, let $K_{\sX/C}$ be the relative canonical divisor of $\sX/C$, that is, 
$K_{\sX}$ minus the pullback of $K_C$. Moreover, for any $t\in C$, let $\sX_t$ be the 
scheme-theoretic fiber of $\sX\to C$ at $t\in C$.

\begin{ack}
The author thanks Doctor Joonyeong Won 
for many discussions and comments, and Professor Yuji Odaka for comments 
on \S \ref{GIT_section}. 
This work was supported by JSPS KAKENHI Grant Number JP16H06885.
\end{ack}

\section{K-stability of log Fano pairs}\label{K_section}

In this section, we always assume that $(X, \Delta)$ is an $n$-dimensional log Fano pair 
and $L:=-(K_X+\Delta)$.

\subsection{Test configurations}\label{tc_section}

In this section, we recall the definitions for uniform K-stability, K-stability, 
K-polystability and 
K-semistability of log Fano pairs. For detail, see \cite{tian, don, RT, LX, BHJ, dervan} 
and references therein.

\begin{definition}\label{tc_dfn}
\begin{enumerate}
\renewcommand{\theenumi}{\arabic{enumi}}
\renewcommand{\labelenumi}{(\theenumi)}
\item\label{tc_dfn1}
A \emph{test configuration} $(\sX,\sL)/\A^1$ of $(X, L)$ consists of: 
\begin{itemize}
\item
a normal variety $\sX$ together with a projective surjective morphism $\sX\to\A^1$, 
\item
a $\Q$-line bundle $\sL$ on $\sX$ which is ample over $\A^1$, 
\item
an action $\G_m\curvearrowright(\sX,\sL)$ such that the action commutes 
with the multiplicative action $\G_m\curvearrowright\A^1$, and 
\item
a $\G_m$-equivariant isomorphism $(\sX\setminus\sX_0, \sL|_{\sX\setminus\sX_0})
\simeq(X\times(\A^1\setminus\{0\}), p_1^*L)$, where 
$\G_m\curvearrowright(X\times(\A^1\setminus\{0\}), p_1^*L)$ is acting trivially along 
$(X, L)$.
\end{itemize}
\item\label{tc_dfn2}
Let $(\sX, \sL)/\A^1$ be a test configuration of $(X, L)$. 
\begin{enumerate}
\renewcommand{\theenumii}{\roman{enumii}}
\renewcommand{\labelenumii}{(\theenumii)}
\item\label{tc_dfn21}
Let $(\bar{\sX}, \bar{\sL})/\pr^1$ be the compactification of $(\sX, \sL)/\A^1$ gluing 
$(\sX, \sL)/\A^1$ and $(X\times(\pr^1\setminus\{0\}), p_1^*L)/(\pr^1\setminus\{0\})$ 
along the above isomorphism $(\sX\setminus\sX_0, \sL|_{\sX\setminus\sX_0})
\simeq(X\times(\A^1\setminus\{0\}), p_1^*L)$.
\item\label{tc_dfn22}
For any $\Q$-divisor $\Gamma$ on $X$, let $\Gamma_{\sX}$ (resp., 
$\Gamma_{\bar{\sX}}$) be the closure of $\Gamma\times(\A^1\setminus\{0\})$ on 
$\sX$ (resp., on $\bar{\sX}$).
\item\label{tc_dfn23}
$(\sX, \sL)/\A^1$ is said to be \emph{trivial} (resp., \emph{a product-type test 
configuration of $((X, \Delta), L)$}) if $(\sX, \sL)/\A^1$ is $\G_m$-equivalently isomorphic 
to $(X\times\A^1, p_1^*L)$ with the natural $\G_m$-action (resp., if $(\sX, \Delta_{\sX})$ 
is isomorphic to $(X\times\A^1, \Delta\times\A^1)$).
\end{enumerate}
\end{enumerate}
\end{definition}

\begin{definition}\label{DF_dfn}
Let $(\sX, \sL)/\A^1$ be a test configuration of $(X, L)$. 
\begin{enumerate}
\renewcommand{\theenumi}{\arabic{enumi}}
\renewcommand{\labelenumi}{(\theenumi)}
\item\label{DF_dfn1}\cite{wang, odk, LX, BHJ}
The \emph{Donaldson-Futaki invariant} $\DF_\Delta(\sX, \sL)$ is defined to be: 
\[
\DF_\Delta(\sX, \sL):=\frac{1}{(L^{\cdot n})}\left(\frac{n}{n+1}
\left(\bar{\sL}^{\cdot n+1}\right)+\left(\bar{\sL}^{\cdot n}\cdot 
\left(K_{\bar{\sX}/\pr^1}+\Delta_{\bar{\sX}}\right)\right)\right).
\]
\item\label{DF_dfn2}\cite{BHJ} (see also \cite{dervan})
Let 
\[\xymatrix{
& \bar{\sY}  \ar[dl]_\Theta \ar[dr]^\Pi & \\
\bar{\sX} & & X\times\pr^1
}\]
be the normalization of the graph of the natural birational map 
$\bar{\sX}\dashrightarrow X\times\pr^1$. We set 
\[
J^{\NA}(\sX, \sL):=\frac{1}{(L^{\cdot n})}\left(\left(
\Pi^*p_1^*L^{\cdot n}\cdot\Theta^*\bar{\sL}\right)-\frac{1}{n+1}
\left(\bar{\sL}^{\cdot n+1}\right)\right).
\]
\end{enumerate}
\end{definition}

\begin{definition}\label{stability_dfn}
\begin{enumerate}
\renewcommand{\theenumi}{\arabic{enumi}}
\renewcommand{\labelenumi}{(\theenumi)}
\item\label{stability_dfn1}
$(X, \Delta)$ is said to be \emph{K-semistable} (resp., \emph{K-stable}) 
if $\DF_\Delta(\sX, \sL)\geq 0$ (resp., $\DF_\Delta(\sX, \sL)>0$) holds for 
any nontrivial test configuration $(\sX, \sL)/\A^1$ of $(X, L)$. 
\item\label{stability_dfn2}
$(X, \Delta)$ is said to be \emph{uniformly K-stable} 
if there exists $\delta\in(0,1)$ such that 
$\DF_\Delta(\sX, \sL)\geq \delta\cdot J^{\NA}(\sX, \sL)$ holds for 
any test configuration $(\sX, \sL)/\A^1$ of $(X, L)$. 
\item\label{stability_dfn3}
$(X, \Delta)$ is said to be \emph{K-polystable} 
if $(X, \Delta)$ is K-semistable and 
$\DF_\Delta(\sX, \sL)=0$ holds for 
a test configuration $(\sX, \sL)/\A^1$ of $(X, L)$ only if $(\sX, \sL)/\A^1$ is a 
product-type test configuration of $((X, \Delta), L)$. 
\end{enumerate}
\end{definition}

\begin{remark}\label{K_dfn_rmk}
By \cite[Theorem 1.3]{dervan} and \cite[Corollary B]{BHJ}, uniform K-stability implies 
K-stability. Moreover, it is obvious from the definition that K-stability implies 
K-polystability, and K-polystability implies K-semistability. 
\end{remark}

The following theorem of Li and Xu 
is important for the study of K-stability of log Fano pairs. 

\begin{thm}[{\cite[Theorem 7]{LX}, see also \cite[\S 6]{fjt}}]\label{LX_thm}
In order to check those stability conditions in Definition \ref{stability_dfn}, it is enough to 
consider test configurations $(\sX, \sL)/\A^1$ with $\sX_0$ integral. 
\end{thm}

We recall the following lemma. 

\begin{lemma}[{\cite[Lemma 2.2]{B}}]\label{B_lemma}
Let $(\sX, \sL)/\A^1$ be a test configuration of $(X, L)$ with $\sX_0$ integral. 
Then $\sL\sim_\Q-(K_{\sX/\A^1}+\Delta_{\sX})$.
\end{lemma}

\begin{proof}
For any test configuration $(\sX, \sL)/\A^1$ of $(X, L)$, we can take a $\Q$-divisor 
$D$ supported on $\Supp\sX_0$ such that 
$D\sim_\Q-\sL-(K_{\sX/\A^1}+\Delta_{\sX})$. If $\sX_0$ is integral, then $D$ must be  
some multiple of $\sX_0$. 
\end{proof}

\subsection{A valuative criterion}\label{vst_section}

We recall an interpretation for uniform K-stability, K-semistability of log Fano pairs  
introduced in \cite{li, fjt, fjt17}. 

\begin{definition}[{see \cite{fjt17}}]\label{vol_dfn}
Let $F$ be a prime divisor 
over $X$. We fix a projective birational morphism $f\colon\tilde{X}\to X$ with 
$\tilde{X}$ normal such that $F$ is realized as a prime divisor on $\tilde{X}$. 
\begin{enumerate}
\renewcommand{\theenumi}{\arabic{enumi}}
\renewcommand{\labelenumi}{(\theenumi)}
\item\label{vol_dfn1}
We set $A_{(X, \Delta)}(F):=\ord_F(K_{\tilde{X}}-f^*(K_X+\Delta))+1$. 
(Of course, we can define $A_{(X, \Delta)}(F)$ not only for a log Fano pair $(X, \Delta)$ but also 
for any log pair $(X, \Delta)$ with $K_X+\Delta$ $\Q$-Cartier.)
\item\label{vol_dfn2}
For any $r\in\Z_{>0}$ with $rL$ Cartier and for any $k\in\Z_{\geq 0}$, we set 
\[
V_k^r:=H^0(X, krL).
\]
Moreover, for any $j\in\R$, we define the $\C$-vector subspace $\sF_F^jV_k^r$ of 
$V_k^r$ as 
\[
\sF_F^jV_k^r:=\begin{cases}
H^0(\tilde{X}, f^*(krL)+\lfloor -jF\rfloor) & \text{if } j\geq 0,\\
V_k^r & \text{if } j<0.
\end{cases}\]
The definition is independent of the choice of the morphism $f$. 
\item\label{vol_dfn3}
For any $x\in\R_{\geq 0}$, we set 
\[
\vol_X(L-xF):=\vol_{\tilde{X}}(f^*L-xF).
\]
(For the definition of $\vol_{\tilde{X}}$, see \cite{L1, L2}.) We know that 
$\vol_X(L-xF)$ is continuous and non-increasing over $x\in[0,+\infty)$. 
Moreover, if $x\gg 0$, then $\vol_X(L-xF)$ is identically equal to zero since 
$f^*L-xF$ is not pseudo-effective for $x\gg 0$.
\item\label{vol_dfn4}
We set 
\[
\hat{\beta}_{(X, \Delta)}(F):=1-\frac{\int_0^\infty\vol_X(L-xF)dx}
{A_{(X, \Delta)}(F)\cdot (L^{\cdot n})}.
\]
\item\label{vol_dfn5}
$F$ is said to be \emph{dreamy over} $(X, \Delta)$ if the graded $\C$-algebra 
\[
\bigoplus_{k,j\in\Z_{\geq 0}}\sF_F^jV_k^r
\]
is finitely generated for some (hence, for any) $r\in\Z_{>0}$ with $rL$ Cartier. 
\end{enumerate}
We note that the definition of $\hat{\beta}_{(X, \Delta)}(F)$ is independent of 
the choice of the morphism $f$. 
\end{definition}

We use the following valuative criterion for uniform K-stability, K-semistability 
of log Fano pairs in order to prove Theorem \ref{mainthm}. 

\begin{thm}\label{beta_thm}
\begin{enumerate}
\renewcommand{\theenumi}{\arabic{enumi}}
\renewcommand{\labelenumi}{(\theenumi)}
\item\label{beta_thm1}
$($\cite{li, fjt}$)$
The following are equivalent: 
\begin{enumerate}
\renewcommand{\theenumii}{\roman{enumii}}
\renewcommand{\labelenumii}{(\theenumii)}
\item\label{beta_thm11}
$(X, \Delta)$ is K-semistable, 
\item\label{beta_thm12}
$\hat{\beta}_{(X, \Delta)}(F)\geq 0$ holds for any prime divisor $F$ over $X$, 
\item\label{beta_thm13}
$\hat{\beta}_{(X, \Delta)}(F)\geq 0$ holds for any dreamy prime divisor $F$ 
over $(X,\Delta)$. 
\end{enumerate}
\item\label{beta_thm2}
$($\cite{fjt, fjt17}$)$
The following are equivalent: 
\begin{enumerate}
\renewcommand{\theenumii}{\roman{enumii}}
\renewcommand{\labelenumii}{(\theenumii)}
\item\label{beta_thm21}
$(X, \Delta)$ is uniformly K-stable, 
\item\label{beta_thm22}
there exists $\varepsilon>0$ such that $\hat{\beta}_{(X, \Delta)}(F)\geq \varepsilon$ holds 
for any prime divisor $F$ over $X$, 
\item\label{beta_thm23}
there exists $\varepsilon>0$ such that $\hat{\beta}_{(X, \Delta)}(F)\geq \varepsilon$ holds 
for any dreamy prime divisor $F$ over $(X, \Delta)$. 
\end{enumerate}
\end{enumerate}
\end{thm}

Moreover, in Theorem \ref{stps_thm}, we give a valuative criterion for \emph{K-stability}, 
\emph{K-polystability} of log Fano pairs.

\section{Dreamy prime divisors}\label{dreamy_section}

\subsection{Divisorial valuations}\label{divisorial_section}

\begin{definition}\label{divisorial_dfn}
Let $X$ be an $n$-dimensional normal variety with the function field $\C(X)$. 
A \emph{divisorial valuation} on $X$ is a group homomorphism $c\cdot\ord_F\colon 
\C(X)^*\to(\Q, +)$ with $c\in\Q_{>0}$ and $F$ a prime divisor over $X$. For any 
$\Q$-divisor $\Delta$ on $X$ with $K_X+\Delta$ $\Q$-Cartier, we set 
$A_{(X, \Delta)}(c\cdot\ord_F):=c\cdot A_{(X, \Delta)}(F)$. 
\end{definition}

\begin{example}[{see \cite{JM} for example}]\label{q-monomial_ex}
Let us set $X:=\A^n_{x_1,\dots,x_n}$ and take 
$\vec{a}=(a_1,\dots,a_n)\in\Z^n_{\geq 0}\setminus\{(0,\dots,0)\}$. 
For any 
\[
f=\sum_{\vec{\alpha}\in\Z_{\geq 0}^n}f_{\vec{\alpha}}x^{\vec{\alpha}}\in\C[x_1,\dots,x_n]
\setminus\{0\}
\]
with $f_{\vec{\alpha}}\in\C$, let us set 
\[
v(f):=\min\{(\vec{\alpha}\cdot \vec{a})\,|\, f_{\vec{\alpha}}\neq 0\}, 
\]
and we naturally extend $v\colon \C(x_1,\dots,x_n)^*\to\Z$ as a valuation. 
We call the $v$ the \emph{quasi-monomial valuation on} $X$ \emph{for coordinates} 
$(x_1,\dots, x_n)$ \emph{with weights} $(a_1,\dots, a_n)$. 
It is well-known that the valuation $v$ is a divisorial 
valuation on $X$. 
In fact, if $\gcd(a_1,\dots,a_n)=1$, then $v$ is equal to $\ord_F$, where $F$ is 
the exceptional divisor of the weighted blowup of $X$ for coordinates 
$(x_1,\dots,x_n)$ with weights $(a_1,\dots,a_n)$. 
\end{example}

From now on, let $(X, \Delta)$ be a log Fano pair and we set $L:=-(K_X+\Delta)$. 
We recall that a test configuration $(\sX, \sL)/\A^1$ of $(X, L)$ with $\sX_0$ integral 
induces a divisorial valuation on $X$. 

\begin{thm}[{see \cite[Theorem 6.13 and Claim 4]{fjt}}]\label{integral_thm}
Let $(\sX, \sL)/\A^1$ be a nontrivial test configuration of $(X, L)$ with $\sX_0$ integral. 
Take any $r\in\Z_{>0}$ with $r\sL$Cartier. 
\begin{enumerate}
\renewcommand{\theenumi}{\arabic{enumi}}
\renewcommand{\labelenumi}{(\theenumi)}
\item\label{integral_thm1}
In this situation, it is obvious that $\C(\sX)=\C(X)(t)$. Let us consider the restriction 
$v_{\sX_0}\colon\C(X)^*\to\Z$ 
of $\ord_{\sX_0}\colon \C(X)(t)^*\to\Z$ to $\C(X)^*\subset\C(X)(t)^*$. 
Then there exist $c\in\Z_{>0}$ and a dreamy prime divisor $F$ over $(X, \Delta)$ 
such that $v_{\sX_0}=c\cdot\ord_F$ holds. $($In particular, $v_{\sX_0}$ is a divisorial 
valuation on $X$.$)$ 
\item\label{integral_thm2}
We have 
\[
\DF_\Delta(\sX, \sL)=c\cdot\hat{\beta}_{(X, \Delta)}(F)\cdot A_{(X, \Delta)}(F).
\]
\item\label{integral_thm3}$($see also \cite[Proposition 2.15]{BHJ}$)$
There exists $d\in\Q$ with $dr\in\Z$ such that, for any $k\in\Z_{\geq 0}$, we have 
\[
H^0(\sX, kr\sL)=\bigoplus_{j\in\Z}t^{-j}\cdot\sF_F^{\frac{dkr+j}{c}}V_k^r
\]
as a $\C[t]$-module. 
\end{enumerate}
\end{thm}

\begin{proof}
By Lemma \ref{B_lemma}, there exists $d_0\in\Q$ such that 
$\sL=-(K_{\sX/\A^1}+\Delta_{\sX})-d_0\sX_0$ holds. 
Thus the assertion immediately follows 
from \cite[Theorem 6.13 and Claim 4]{fjt} and \cite[\S 2.5]{BHJ}. 
\end{proof}

\begin{definition}\label{aut_dfn}
We set 
\[
\Aut(X, \Delta; L):=\{\sigma\in\Aut(X; L)\,|\, \sigma_*\Delta=\Delta\}\subset\Aut(X; L)
\]
as in \cite[Proposition 4.6]{ambro}. For any $\sigma\in\Aut(X)$ with 
$\sigma_*\Delta=\Delta$, we have 
$\sigma^*L\simeq L$. Thus we write $\Aut(X, \Delta):=\Aut(X, \Delta; L)$. 
\end{definition}

\begin{example}\label{aut_ex}
Let $\rho\colon\G_m\times X\to X$ with $(t,x)\mapsto \rho_t(x)$ be a one-parameter 
subgroup of $\Aut(X; L)$ (that is, $\rho$ is given by a nontrivial morphism 
$\rho\colon \G_m\to\Aut(X; L)$ in the category of algebraic groups). 
Then, as in \cite[Example 2.3]{BHJ}, the diagonal action $\G_m\curvearrowright 
(X, L)\times\A^1$ gives a test configuration $(\sX, \sL)/\A^1$ of $(X, L)$. 
In particular, if $\rho$ is a one-parameter subgroup of $\Aut(X, \Delta)$, then 
this is a product-type test configuration of $((X, \Delta), L)$. Conversely, any 
nontrivial product test configuration of $((X, \Delta), L)$ is recovered from 
some one parameter subgroup $\G_m\to \Aut(X, \Delta)$. 
See \cite[Proposition 3.7]{RT} for example.

Let 
\[
\rho^*\colon \C(X)\to \C(X)(t)
\]
be the field extension obtained by the morphism $\rho$. 
Since the $\G_m$-equivariant 
isomorphism 
\[
\phi\colon(\sX\setminus\sX_0, \sL|_{\sX\setminus\sX_0})
\to(X\times(\A^1\setminus\{0\}), p_1^*L)
\]
is given by 
\begin{eqnarray*}
\phi\colon X\times (\A^1\setminus\{0\}) & \to & X\times (\A^1\setminus\{0\})\\ 
(x,t) & \mapsto &(\rho_{t^{-1}}(x),t), 
\end{eqnarray*}
the divisorial valuation $v_{\sX_0}\colon \C(X)^*\to\Z$ is given by 
\[
v_{\sX_0}(f)=\ord_{(t^{-1})}\rho^*f
\]
for any $f\in\C(X)^*$.
\end{example}

\begin{example}\label{aut_P_ex}
We consider a special case of Example \ref{aut_ex}. 
We assume that $X$ is equal to $\pr^n$ with homogeneous coordinates 
$z_0:\cdots:z_n$. Then $\Aut(X)=\Aut(X; L)=\PGL(n+1)$. 
Take any $(a_1,\dots,a_n)\in\Z_{\geq 0}^n\setminus\{(0,\dots,0)\}$. 
Let us consider the one-parameter subgroup 
\begin{eqnarray*}
\rho\colon \G_m & \to & \PGL(n+1) \\
t & \mapsto & \diag(1, t^{-a_1},\dots,t^{-a_n}).
\end{eqnarray*}
More precisely, the morphism $\rho\colon \G_m\times X\to X$ is given by 
\[
(t; z_0:\cdots:z_n)\mapsto (z_0:t^{-a_1}z_1:\cdots:t^{-a_n}z_n).
\]
Set $x_i:=z_i/z_0$. Then 
\[
\rho^*\colon \C(x_1,\dots,x_n)\to\C(x_1,\dots,x_n)(t)
\]
is obtained by $x_i\mapsto t^{-a_i}x_i$. Moreover, the isomorphism 
$\phi\colon X\times(\A^1\setminus\{0\})\to X\times(\A^1\setminus\{0\})$ in Example 
\ref{aut_ex} is obtained by 
\[
(z_0:\cdots:z_n;t)\mapsto(z_0:t^{a_1}z_1:\cdots:t^{a_n}z_n;t).
\]
Let $(\sX, \sL)/\A^1$ be the test configuration of $(X, L)$ obtained by $\rho$ 
as in Example \ref{aut_ex}. Then the divisorial valuation $v_{\sX_0}$ is the 
quasi-monomial valuation on 
\[
\A^n_{x_1,\dots,x_n}=\pr^n\setminus(z_0=0)\subset X
\]
for coordinates $(x_1,\dots,x_n)$ with weights $(a_1,\dots,a_n)$. 
\end{example}

\subsection{Product-type prime divisors}\label{prod_section}

In this section, we always assume that $(X, \Delta)$ is an $n$-dimensional log Fano 
pair, $L:=-(K_X+\Delta)$, and $F$ is a dreamy prime divisor over $(X, \Delta)$. 

\begin{definition}\label{drtc_dfn}
For any $c\in\Z_{>0}$ and for any $r\in\Z_{>0}$ with $rL$ Cartier, let us set 
the test configuration $(\sX^{F,c}, \sL^{F,c})/\A^1$ of $(X, L)$ defined by 
\begin{eqnarray*}
\sX^{F,c} & := & \Proj_{\A^1_t}\bigoplus_{k\in\Z_{\geq 0}}\left(
\bigoplus_{j\in\Z}t^{-j}\cdot\sF_F^{\frac{j}{c}}V_k^r\right), \\
\sL^{F,c} & := & \frac{1}{mr}(\text{relative } \sO(m))\quad \text{for a sufficiently 
divisible }m\in\Z_{>0}, 
\end{eqnarray*}
with the natural $\G_m$-action as in \cite[\S 1.2 and Proposition 2.15]{BHJ}. 
(We will see in Lemma \ref{normal_lemma} that $\sX^{F,c}$ is normal.)
Moreover, we set $\sX^F:=\sX^{F, 1}$ and $\sL^F:=\sL^{F,1}$. 
\end{definition}

Note that those are independent of the choice of $r$. From the construction, 
$\sX^{F,c}$ is given by the fiber product of 
\[\xymatrix{
\sX^F  \ar[dr] &  & \A^1 \ar[dl]^{\psi_c}\\
& \A^1 & \\
}\]
and $\sL^{F,c}$ is the pullback of $\sL^F$, 
where $\psi_c$ is obtained by $t\mapsto t^c$.

\begin{lemma}\label{normal_lemma}
$(\sX^{F,c}, \sL^{F,c})/\A^1$ is a nontrivial 
test configuration of $(X, L)$ with $\sX_0$ integral. 
\end{lemma}

\begin{proof}
The scheme-theoretic fiber $\sX_0^{F,c}$ of $\sX^{F,c}\to\A^1$ at $0\in\A^1$ is 
isomorphic to 
\[
\Proj\bigoplus_{k\in\Z_{\geq 0}}\left(\bigoplus_{j\in\Z_{\geq 0}}S_{k,j}\right), 
\]
from the construction, where 
\[
S_{k,j}:=\sF_F^jV_k^r/\sF_F^{j+1}V_k^r.
\]
Set 
\[
S:=\bigoplus_{k,j\in\Z_{\geq 0}}S_{k,j}.
\]
It is enough to show that $S$ is an integral domain 
in order to show that $\sX_0^{F,c}$ is integral. Take any 
\[
\bar{f}_i\in S_{k_i,j_i}\setminus\{0\} \,\text{ with }\, f_i\in\sF_F^{j_i}V^r_{k_i}\setminus
\sF_F^{j_i+1}V^r_{k_i}
\]
for $i=1$, $2$. Then $\DIV(f_1 f_2)$ vanishes on $F$ exactly $j_1+j_2$ times. 
Thus $\bar{f}_1\bar{f}_2\in S_{k_1+k_2, j_1+j_2}\setminus\{0\}$. Thus $S$ is an 
integral domain. Thus $\sX^{F, c}$ is normal and $(\sX^{F, c}, \sL^{F,c})/\A^1$ is a 
test configuration of $(X, L)$ by \cite[Proposition 2.6 (iv) and Proposition 2.15]{BHJ}. 
\end{proof}

As a conclusion of the arguments in \S \ref{divisorial_section} and \S \ref{prod_section}, 
together with \cite[Proposition 2.15]{BHJ}, we have the following correspondence: 
\begin{itemize}
\item
For any nontrivial test configuration $(\sX, \sL)/\A^1$ of $(X, L)$ with $\sX_0$ integral, 
there exist $c\in\Z_{>0}$, $d\in\Q$, and a dreamy prime divisor $F$ over $(X, \Delta)$ 
such that $v_{\sX_0}=c\cdot \ord_F$ and $(\sX, \sL)/\A^1$ is isomorphic to 
$(\sX^{F,c}, \sL^{F, c}+d\sX_0^{F,c})/\A^1$. 
\item
Conversely, for any dreamy prime divisor $F$ over $(X, \Delta)$ and for any $c\in\Z_{>0}$, 
the test configuration $(\sX^{F,c}, \sL^{F,c})/\A^1$ of $(X, L)$ in Definition \ref{drtc_dfn} 
satisfies that $\sX_0^{F,c}$ is integral. Moreover, $v_{\sX_0^{F,c}}$ is equal to $c\cdot \ord_F$. 
\end{itemize}

\begin{definition}\label{prod_div_dfn}
A dreamy prime divisor $F$ over $(X, \Delta)$ is said to be \emph{a product-type 
prime divisor over} $(X, \Delta)$ if $(\sX^F, \sL^F)/\A^1$ in Definition \ref{drtc_dfn} 
is a product-type test configuration of $((X, \Delta), L)$. 
\end{definition}

From Example \ref{aut_ex}, we immediately get the following: 

\begin{proposition}\label{dream_prod_prop}
Let $F$ be a prime divisor over $X$. Then the following are equivalent: 
\begin{enumerate}
\renewcommand{\theenumi}{\arabic{enumi}}
\renewcommand{\labelenumi}{(\theenumi)}
\item\label{dream_prod_prop1}
$F$ is a product-type prime divisor over $(X, \Delta)$, 
\item\label{dream_prod_prop2}
there exists a one-parameter subgroup $\rho\colon\G_m\to\Aut(X, \Delta)$ such that 
the valuation $\ord_F$ on $X$ is equal to the valuation $\ord_{(t^{-1})}\circ\rho^*$ 
as in Example \ref{aut_ex}. 
\end{enumerate}
\end{proposition}

From the above observation, 
we get the following valuative criterion for K-stability, K-polystability 
of log Fano pairs. 

\begin{thm}\label{stps_thm}
Let $(X, \Delta)$ be a log Fano pair. 
\begin{enumerate}
\renewcommand{\theenumi}{\arabic{enumi}}
\renewcommand{\labelenumi}{(\theenumi)}
\item\label{stps_thm1}
The following are equivalent: 
\begin{enumerate}
\renewcommand{\theenumii}{\roman{enumii}}
\renewcommand{\labelenumii}{(\theenumii)}
\item\label{stps_thm11}
$(X, \Delta)$ is K-stable, 
\item\label{stps_thm12}
$\hat{\beta}_{(X, \Delta)}(F)>0$ holds for any dreamy prime divisor $F$ over $(X, \Delta)$.
\end{enumerate}
\item\label{stps_thm2}
The following are equivalent: 
\begin{enumerate}
\renewcommand{\theenumii}{\roman{enumii}}
\renewcommand{\labelenumii}{(\theenumii)}
\item\label{stps_thm21}
$(X, \Delta)$ is K-polystable, 
\item\label{stps_thm22}
$\hat{\beta}_{(X, \Delta)}(F)\geq 0$ holds for any dreamy prime divisor $F$ over 
$(X, \Delta)$, and equality holds only if $F$ is a product-type prime divisor over 
$(X, \Delta)$.
\end{enumerate}
\end{enumerate}
\end{thm}

\begin{proof}
Immediately follows from the above argument and Theorems \ref{LX_thm} and 
\ref{integral_thm}. 
\end{proof}

\begin{remark}\label{Kst_rmk}
If $\Delta=0$, then Theorem \ref{stps_thm} \eqref{stps_thm1} is nothing but 
\cite[Theorem 1.6]{fjt} (see also \cite[Theorem 3.7]{li}). We note that the proof of 
Theorem \ref{stps_thm} \eqref{stps_thm1} differs from the proof of 
\cite[Theorem 1.6]{fjt}. 
\end{remark}

\section{Blowups along linear subspaces}\label{blowup_section}

In this section, we see an easy necessary condition for 
log Fano hyperplane arrangements being K-stable or K-semistable. 

\begin{proposition}\label{blowup_prop}
Let $(X, \Delta)$ be an $n$-dimensional log Fano hyperplane arrangement with 
$\Delta\neq 0$. Set 
\[
\Gamma:=\frac{n+1}{d_{(X, \Delta)}}\Delta.
\]
If $(X, \Delta)$ is K-semistable $($resp., K-stable$)$, then $(X, \Gamma)$ is an lc 
$($resp., a klt$)$ Calabi-Yau hyperplane arrangement. 
\end{proposition}

\begin{proof}
Take any $W\in\LL'(X, \Delta)$. 
We set $c:=c^X(W)$, $d^W:=d_{(X, \Delta)}(W)$ and $d:=d_{(X, \Delta)}$ for simplicity. 
Let us consider the blowup 
$\sigma\colon Y\to X=\pr^n$ along $W$. Set $F:=\sigma^{-1}(W)$. 
We note that $Y$ is isomorphic to 
\[
\pr_{\pr^{c-1}}\left(\sO_{\pr^{c-1}}^{\oplus n+1-c}\oplus
\sO_{\pr^{c-1}}(1)\right). 
\]
In particular, $F$ is a dreamy prime divisor over $(X, \Delta)$ since $Y$ is toric. 
Let $\pi\colon Y\to \pr^{c-1}$ be the natural projective space bundle morphism. 
Let us set
\begin{eqnarray*}
\xi &:=& \sigma^*\sO_{\pr^n}(1), \\
A & :=& \pi^*\sO_{\pr^{c-1}}(1). 
\end{eqnarray*}
Then $F\sim\xi-A$. By Theorems \ref{beta_thm} and \ref{stps_thm}, we have 
$\hat{\beta}_{(X, \Delta)}(F)\geq 0$ (resp., $>0$). 
On the other hand, we have 
\begin{eqnarray*}
\hat{\beta}_{(X, \Delta)}(F)&=&1-\frac{\int_0^{n+1-d}\left(\left(
(n+1-d)\xi-x(\xi-A)\right)^{\cdot n}\right)dx}{\left(c
-d^W\right)\cdot (n+1-d)^n}\\
&=&1-\frac{n+1-d}{c-d^W}\int_0^1\left(\left(
(1-x)\xi+xA\right)^{\cdot n}\right)dx\\
&=&\frac{cd-(n+1)d^W}{(n+1)\left(
c-d^W\right)}.
\end{eqnarray*}
This implies that 
\[
\frac{c}{d^W}\geq \frac{n+1}{d} \quad \left(\text{resp., }\,\,\
\frac{c}{d^W}>\frac{n+1}{d}\right). 
\]
Thus we get the assertion from Theorem \ref{mustata_thm}. 
\end{proof}

\section{Log Calabi-Yau hyperplane arrangements of class $\sP$}\label{P_section}

In this section, we define the notion of log Calabi-Yau hyperplane arrangements 
of class $\sP$. 

\begin{definition}\label{projective_dfn}
Let $s\in\Z_{>0}$ and $n_1,\dots,n_s\in\Z_{\geq 0}$. Let us set $n:=\sum_{i=1}^s(n_i+1)-1$. 
Consider the $n_i$-dimensional projective space $\pr^{n_i}$ with 
homogeneous coordinates 
$z_{i0}:\cdots:z_{in_i}$
for any $1\leq i\leq s$. 
\begin{enumerate}
\renewcommand{\theenumi}{\arabic{enumi}}
\renewcommand{\labelenumi}{(\theenumi)}
\item\label{projective_dfn1}
Let $S(\pr^{n_1},\dots,\pr^{n_s})$ be the $n$-dimensional projective space 
with homogeneous coordinates 
\[
z_{10}:\cdots:z_{1n_1}:\cdots\cdots:z_{s0}:\cdots:z_{sn_s}.
\]
(We write $\pr:=S(\pr^{n_1},\dots,\pr^{n_s})$ just for simplicity.)
\item\label{projective_dfn2}
For any subset $I\subset \{1,\dots,s\}$, let $P_I$, $Q_I\subset\pr$ be the linear 
subspaces defined by 
\begin{eqnarray*}
P_I & := & (z_{jk}=0\,|\, j\in\{1,\dots,s\}\setminus I, \, 0\leq k\leq n_j), \\
Q_I & := & (z_{jk}=0\,|\, j\in I, \, 0\leq k\leq n_j).
\end{eqnarray*}
(Note that $Q_I=P_{\{1,\dots,s\}\setminus I}$, $P_\emptyset=\emptyset$ and 
$Q_\emptyset=\pr$.)
Moreover, we set $P_i:=P_{\{i\}}(=\pr^{n_i})$ and $Q_i:=Q_{\{i\}}$ for any $1\leq i\leq s$. 
Obviously, we have 
\begin{eqnarray*}
\dim P_I & = & \sum_{i\in I}(n_i+1)-1, \\
\dim Q_I & = & n-\sum_{i\in I}(n_i+1)
\end{eqnarray*}
and $P_I\cap Q_I=\emptyset$. 
\item\label{projective_dfn3}
Let $W_1,\dots,W_m\subset\pr$ be linear subspaces (we allow $\emptyset$). 
Let $\langle W_1,\dots,W_m\rangle\subset\pr$ be the smallest linear subspace 
containing $W_1\cup\cdots\cup W_m$. 
\item\label{projective_dfn4}
Let $(\pr^{n_1}, \Xi_1),\dots,(\pr^{n_s},\Xi_s)$ be hyperplane arrangements. 
(If $n_i=0$, then we set $\Xi_i:=\emptyset$.) Let $\Xi_i=\sum_{k=1}^{m_i}d_{i,k}\Xi_{i,k}$ 
be the irreducible decomposition. It can be seen that each $\pr^{n_i}$ is a linear subspace 
of $\pr$ via the identification $\pr^{n_i}=P_i$. For each $1\leq i\leq s$, let $\Gamma_i$ be 
the $\Q$-divisor on $\pr$ defined by the equation $\Gamma_i:=\langle\Xi_i, Q_i\rangle$,  
where 
\[
\langle\Xi_i, Q_i\rangle:=\begin{cases}
\sum_{k=1}^{m_i}d_{i,k}\langle\Xi_{i,k},Q_i\rangle & \text{if }\, n_i>0, \\ 
1\cdot Q_i & \text{if }\, n_i=0.
\end{cases}\]
We set $\Gamma:=\sum_{i=1}^s\Gamma_i$. Obviously, $(\pr, \Gamma)$ is a hyperplane 
arrangement. The hyperplane arrangement $(\pr, \Gamma)$ is denoted by 
\[
S((\pr^{n_1}, \Xi_1),\dots,(\pr^{n_s}, \Xi_s)).
\]
\end{enumerate}
\end{definition}

The following two lemmas are easy. 

\begin{lemma}\label{S_lemma}
Under the notion in Definition \ref{projective_dfn} \eqref{projective_dfn4}, we have 
\[
\LL(\pr, \Gamma)=\{\langle W_1,\dots,W_s\rangle\,|\, W_i\in\LL(\pr^{n_i},\Xi_i)\,\,
\text{ for any }\,\,1\leq i\leq s\}.
\]
Moreover, we have 
\begin{eqnarray*}
c^\pr\left(\langle W_1,\dots,W_s\rangle\right)&=&\sum_{i=1}^s c^{\pr^{n_i}}(W_i), \\
d_{(\pr, \Gamma)}\left(\langle W_1,\dots,W_s\rangle\right)
&=&\sum_{i=1}^s d_{(\pr^{n_i}, \Xi_i)}(W_i).
\end{eqnarray*}
\end{lemma}

\begin{proof}
Take any $W_i\in \LL(\pr^{n_i},\Xi_i)$. There exist components 
$\Xi_{i,1},\dots,\Xi_{i,l_i}$ of $\Xi_i$ such that $W_i=\bigcap_{k=1}^{l_i}\Xi_{i,k}$. 
Thus we get 
$\langle W_1,\dots,W_s\rangle\in\LL(\pr, \Gamma)$ 
since $\langle \Xi_{i,k}, Q_i\rangle$ is a component of $\Gamma$ and 
\[
\bigcap_{i=1}^s\bigcap_{k=1}^{l_i}\langle \Xi_{i,k}, Q_i\rangle=\bigcap_{i=1}^s
\langle W_i,Q_i\rangle=\langle W_1,\dots,W_s\rangle.
\]

Conversely, take any $W\in\LL(\pr, \Gamma)$. Then we can write 
\[
W=\bigcap_{i=1}^s\bigcap_{k=1}^{l_i}\Gamma_{i,k},
\]
where $\Gamma_{i,k}$ are components of $\Gamma_i$. Set 
$W_i:=\bigcap_{k=1}^{l_i}\Gamma_{i,k}\cap P_i$.
Note that $W_i\in\LL(\pr^{n_i},\Xi_i)$ and 
$\bigcap_{k=1}^{l_i}\Gamma_{i,k}=\langle W_i,Q_i\rangle$.
Thus we have $W=\langle W_1,\dots,W_s\rangle$.
\end{proof}

\begin{lemma}\label{CYempt_lemma}
Let $(X, \Gamma)$ be an lc Calabi-Yau hyperplane arrangement. Let 
$\Gamma=\sum_{i=1}^md_iH_i$ be the irreducible decomposition. 
Then we have $\bigcap_{i=1}^mH_i=\emptyset$.
\end{lemma}

\begin{proof}
Set $n:=\dim X$. If not, then there exists $W\in\LL'(X, \Gamma)$ such that 
$d_{(X, \Gamma)}(W)=n+1$. However, by Theorem \ref{mustata_thm}, we have 
$1\geq d_{(X, \Gamma)}(W)/c^X(W)$. This leads to a contradiction since $c^X(W)\leq n$. 
\end{proof}

\begin{definition}\label{P_dfn}
An $n$-dimensional hyperplane arrangement $(X, \Gamma)$ is said to be 
a \emph{log Calabi-Yau hyperplane arrangement of class $\sP$} if there exist 
$s\in\Z_{>0}$ and $n_1,\dots,n_s\in\Z_{\geq 0}$ with $n+1=\sum_{i=1}^s(n_i+1)$ and 
there exist \emph{klt} Calabi-Yau hyperplane arrangements $(\pr^{n_i}, \Xi_i)$ for all 
$1\leq i\leq s$ such that $(X, \Gamma)$ is isomorphic to 
$S((\pr^{n_1}, \Xi_1),\dots,(\pr^{n_s}, \Xi_s))$.
\end{definition}

\begin{proposition}\label{P_prop}
Let $(X, \Gamma)=S((\pr^{n_1}, \Xi_1),\dots,(\pr^{n_s}, \Xi_s))$ be a log 
Calabi-Yau hyperplane arrangement of class $\sP$ as in Definition \ref{P_dfn}. 
\begin{enumerate}
\renewcommand{\theenumi}{\arabic{enumi}}
\renewcommand{\labelenumi}{(\theenumi)}
\item\label{P_prop1}
$(X, \Gamma)$ is an lc Calabi-Yau hyperplane arrangement. 
\item\label{P_prop2}
The set of lc centers of $(X, \Gamma)$ is equal to the set 
\[
\{P_I\,|\,\emptyset\subsetneq I\subsetneq\{1,\dots,s\}\}, 
\]
where an \emph{lc center} of $(X, \Gamma)$ is defined to be the image on $X$ 
of a prime divisor $F$ over $X$ with $A_{(X, \Gamma)}(F)=0$. 
\item\label{P_prop3}
Take any $\emptyset\subsetneq I\subsetneq\{1,\dots,s\}$. Assume that a prime divisor 
$F$ over $X$ satisfies that $A_{(X, \Gamma)}(F)=0$ and the image of $F$ on $X$ 
is equal to $P_I$. Fix $i_0\in I$ and set $x_{ij}:=z_{ij}/z_{i_00}$. Then, for any 
$i\in\{1,\dots,s\}\setminus I$, there exists $a_i\in\Z_{>0}$ with 
$\gcd(a_i)_{i\in\{1,\dots,s\}\setminus I}=1$ such that the valuation 
$\ord_F$ is the quasi-monomial valuation on $X\setminus (z_{i_00}=0)$ for coordinates 
\[
(x_{ij})_{1\leq i\leq s, 0\leq j\leq n_i, (i,j)\neq (i_0,0)}
\]
with weights 
\[\begin{cases}
a_i & \text{if }\, i\in\{1,\dots,s\}\setminus I, \\
0 & \text{otherwise}.
\end{cases}\]
\end{enumerate}
\end{proposition}

\begin{proof}
\eqref{P_prop1}
Since $d_{(\pr^{n_i}, \Xi_i)}=n_i+1$, we have $d_{(X, \Gamma)}=n+1$. Take any 
$W\in\LL'(X, \Gamma)$. By Lemma \ref{S_lemma}, there exists 
$W_i\in\LL(\pr^{n_i}, \Xi_i)$
for any $1\leq i\leq s$ such that $W=\langle W_1,\dots,W_s\rangle$. 
Since $(\pr^{n_i}, \Xi_i)$ is klt, we have $d_{(\pr^{n_i}, \Xi_i)}(W_i)<c^{\pr^{n_i}}(W_i)$ if 
$\emptyset\subsetneq W_i\subsetneq \pr^{n_i}$. On the other hand, we have 
$d_{(\pr^{n_i}, \Xi_i)}(\emptyset)=c^{\pr^{n_i}}(\emptyset)=n_i+1$ and 
$d_{(\pr^{n_i}, \Xi_i)}(\pr^{n_i})=c^{\pr^{n_i}}(\pr^{n_i})=0$. Thus, by Lemma \ref{S_lemma}, 
we have 
\[
\frac{d_{(X, \Gamma)}(W)}{c^X(W)}=\frac{\sum_{i=1}^s d_{(\pr^{n_i}, \Xi_i)}(W_i)}{\sum_{i=1}^s 
c^{\pr^{n_i}}(W_i)}\leq 1.
\]
Thus $(X, \Gamma)$ is lc by Theorem \ref{mustata_thm}. 

\eqref{P_prop2}
Note that $W\in\LL'(X, \Gamma)$ is an lc center of $(X, \Gamma)$ if and only if the above 
inequality is equal. The condition is equivalent to the condition $W_i\in\{\emptyset$, $
\pr^{n_i}\}$ for all $1\leq i\leq s$. The condition is nothing but $W=P_I$ for some 
$\emptyset\subsetneq I\subsetneq \{1,\dots,s\}$. 

\eqref{P_prop3}
For any $0\leq m\leq n-1$, we set 
\[
\sW_m:=\{P_I\,|\,I\subset\{1,\dots,s\}, \, \dim P_I=m\}. 
\]
Let us consider the sequence of blowups 
\[
Y:=X_n\to X_{n-1}\to\cdots\to X_1\to X_0=X
\]
such that $X_{m+1}\to X_m$ is the blowup along the strict transform of 
$\bigcup_{P_I\in\sW_m}P_I$ (this is a smooth subscheme of $X_m$). 
Let $D\subset Y$ be the reduced simple normal crossing divisor given by 
the inverse image of $\bigcup_{I\subsetneq\{1\dots,s\}}P_I$. 
Then, by \cite[Lemma 2.1]{teitler} (see also \cite[Lemma 2.29 and Corollary 2.31]{KoMo}), 
any prime divisor $F$ over $X$ with 
$A_{(X, \Gamma)}(F)=0$ is given by the composition of toroidal blowups with respects 
to $(Y, D)$ (in the sense of \cite[2.16 (2)]{BAB}). 
In particular, since each $P_I$ is torus invariant, after perturbing $\pr^{n_1},\dots,\pr^{n_s}$ 
and homogeneous coordinates $z_{i0}:\cdots:z_{in_i}$ 
if necessary, we may assume that 
the image of $F$ on $X$ is equal to $P_{\{1,\dots,r\}}$ for some $1\leq r\leq s-1$ and 
$\ord_F$ is equal to the quasi-monomial 
valuation on $X\setminus(z_{10}=0)$ for coordinates 
$(x_{ij}:=z_{ij}/z_{10})$ with weights 
\[
\begin{cases}
0 & \text{if }\, 1\leq i\leq r, \\
a_{i,j} & \text{if }\, r+1\leq i\leq s,
\end{cases}\]
where 
\begin{eqnarray*}
&&a_{i,0}\leq\cdots\leq a_{i,n_i}\quad\text{for all }\,r+1\leq i\leq s, \\
&&0<a_{r+1, 0}\leq\cdots\leq a_{s,0}, \\
&&\gcd(a_{i,j})=1.
\end{eqnarray*}
In this situation, we have $A_{(X, 0)}(F)=\sum_{r+1\leq i\leq s, 0\leq j\leq n_i}a_{i,j}$. 
If $1\leq i\leq r$, then obviously $\ord_F\Gamma_i=0$ holds. 
Assume that $r+1\leq i\leq s$ and $a_{i,0}=\cdots=a_{i,n_i}$. Then we have 
\[
\ord_F\Gamma_i=(n_i+1)a_{i,0}=\sum_{0\leq j\leq n_i}a_{i,j}. 
\]

From now on, we assume that $a_{i,0}<a_{i,n_i}$. Write 
$\Xi_i=\sum_{k=1}^{m_i}d_{i,k}\Xi_{i,k}$, $\Gamma_i=\sum_{k=1}^{m_i}d_{i,k}\Gamma_{i,k}$ 
with $\Gamma_{i,k}:=\langle\Xi_{i,k}, Q_i\rangle$. 
Assume that the defining equation of $\Xi_{i,k}\subset\pr^{n_i}$ is 
\[
h_{i,0}^kz_{i0}+\cdots+h_{i,n_i}^kz_{in_i}=0.
\]
Then 
\[
\ord_F\Gamma_{i,k}=a_{i,j_k}:=\min_{0\leq j\leq n_i}\{a_{i,j}\,|\, h_{i,j}^k\neq 0\}
\]
and $\ord_F\Gamma_i=\sum_{k=1}^{m_i}d_{i,k}a_{i,j_k}$.

Let us consider the quasi-monomial valuation $v_i$ on $\pr^{n_i}\setminus(z_{i0}=0)$ 
for coordinates $(y_{ij}:=z_{ij}/z_{i0})_{1\leq j\leq n_i}$ with weights 
$(a_{i,1}-a_{i,0},\dots,a_{i,n_i}-a_{i,0})$. Then 
\begin{eqnarray*}
A_{(\pr^{n_i}, 0)}(v_i) & = & \sum_{j=1}^{n_i}(a_{i,j}-a_{i,0})=\sum_{j=0}^{n_i}a_{i,j}-(n_i+1)a_{i,0}, \\
v_i(\Xi_{i,k}) & = & a_{i, j_k}-a_{i,0}, \\
A_{(\pr^{n_i}, \Xi_i)}(v_i) & = & \sum_{j=0}^{n_i}a_{i,j}-\ord_F\Gamma_i.
\end{eqnarray*}
Since $0<A_{(\pr^{n_i}, \Xi_i)}(v_i)$, we have $\sum_{j=0}^{n_i}a_{i,j}>\ord_F\Gamma_i$. 
However, we know that 
\[
A_{(X, \Gamma)}(F)=\sum_{i=r+1}^s\left(\sum_{0\leq j\leq n_i}a_{i,j}-\ord_F\Gamma_i\right)
\]
is equal to zero. This leads to a contradiction. Therefore we have 
$a_{i,0}=\cdots=a_{i,n_i}$ for any $r+1\leq i\leq s$. 
\end{proof}

\begin{proposition}\label{suff_prop}
Let $(X, \Gamma)$ be an $n$-dimensional log Calabi-Yau hyperplane arrangement 
of class $\sP$, let $d\in\Q\cap(0, n+1)$ and let $F$ be a prime divisor over $X$ with 
$A_{(X, \Gamma)}(F)=0$. Set 
\[
\Delta:=\frac{d}{n+1}\Gamma. 
\]
Then $(X, \Delta)$ is a log Fano pair and $F$ is a product-type prime divisor over 
$(X, \Delta)$.
\end{proposition}

\begin{proof}
It is obvious from Proposition \ref{P_prop} such that $(X, \Delta)$ is a log Fano pair 
and $F$ is a dreamy prime divisor over $(X, \Delta)$. 
As in Proposition \ref{P_prop}, we may assume that 
$(X, \Gamma)=S((\pr^{n_1}, \Xi_1),\dots,(\pr^{n_s}, \Xi_s))$ with $(\pr^{n_i}, \Xi_i)$ 
klt Calabi-Yau hyperplane arrangements for all $1\leq i\leq s$, the image of $F$ on $X$ is 
equal to $P_{\{1,\dots,r\}}$ for some $1\leq r\leq s-1$ and $\ord_F$ is the 
quasi-monomial valuation on $X\setminus(z_{10}=0)$ for coordinates $(x_{ij}:=z_{ij}/z_{10})$ 
with weights 
\[\begin{cases}
0 & \text{if }\, 1\leq i\leq r, \\
a_i & \text{if }\, r+1\leq i\leq s
\end{cases}\]
with $0<a_{r+1}\leq\cdots\leq a_s$ and $\gcd(a_i)=1$. 
Then, by Example \ref{aut_P_ex}, $F$ is given by the one-parameter subgroup 
\begin{eqnarray*}
\rho\colon\G_m&\to&\Aut(X)\\ 
t&\mapsto&
\diag(1,\dots,1,\underbrace{t^{-a_{r+1}},\dots,t^{-a_{r+1}}}_{n_{r+1}+1 
\text{ times}},\dots\dots,\underbrace{t^{-a_s},\dots,t^{-a_s}}_{n_s+1 
\text{ times}}) 
\end{eqnarray*}
of $\Aut(X)$. It is enough to show that $\rho$ factors through $\Aut(X, \Delta)$. 
In fact, for all $t\in\G_m$ and for all $1\leq i\leq s$, 
if we set $\Gamma_i:=\langle\Xi_i, Q_i\rangle$, we get $\rho_t^*\Gamma_i=\Gamma_i$. 
In particular, we have $\rho_t^*\Gamma=\Gamma$
\end{proof}

We give a combinatorial characterization of log Calabi-Yau hyperplane arrangement of 
class $\sP$. 

\begin{thm}\label{comb_thm}
Let $(X, \Gamma)$ be an $n$-dimensional lc Calabi-Yau hyperplane arrangement. 
Assume that, for any lc center $P\subset X$ of $(X, \Gamma)$ with $\dim P=c-1$, there 
exists a linear subspace $Q=Q(P)\subset X$ with $\dim Q=n-c$ and 
$P\cap Q=\emptyset$ such that $d_{(X, \Gamma)}(Q)=c$. 
Then $(X, \Gamma)$ is a log Calabi-Yau hyperplane arrangement of class $\sP$.
\end{thm}

\begin{proof}
\textbf{Step 1.}
We firstly note that there is no hyperplane $H$ on $X$ with $P$, $Q\subset H$. 
Write $\Gamma=\sum_H d_H H$. We set 
$\Gamma_P:=\sum_{P\not\subset H}d_H H$ and $\Gamma_Q:=\sum_{P\subset H}d_H H$. 
Obviously, we have $\Gamma=\Gamma_P+\Gamma_Q$, and $\Gamma_P$, $\Gamma_Q$ 
are uniquely determined by $P$. Since $\Gamma_P$ and $\Gamma_Q$ have no 
common irreducible component, $d_{(X, \Gamma_P)}\geq c$ and 
$d_{(X, \Gamma_Q)}\geq n+1-c$, any component of $\Gamma_P$ contains $Q$. 
Set $\Xi_P:=\Gamma_P\cap P$ and $\Xi_Q:=\Gamma_Q\cap Q$. Then both 
$(P, \Xi_P)$ and $(Q, \Xi_Q)$ are log Calabi-Yau hyperplane arrangements and 
$\Gamma_P=\langle \Xi_P, Q\rangle$, $\Gamma_Q=\langle \Xi_Q, P\rangle$. 
In particular, we have $(X, \Gamma)=S((P, \Xi_P), (Q, \Xi_Q))$. 

\textbf{Step 2.}
Take an arbitrary $P'\in\LL'(P, \Xi_P)$. Then, by Lemma \ref{S_lemma}, 
we have $P'$, $\langle P', Q\rangle\in\LL'(X, \Gamma)$. Moreover, we have 
\begin{eqnarray*}
c^X(P')=c^P(P')+n+1-c, && c^X\left(\langle P', Q\rangle\right)=c^P(P'), \\
d_{(X, \Gamma)}(P')=d_{(P, \Xi_P)}(P')+n+1-c, && d_{(X, \Gamma)}\left(
\langle P', Q\rangle\right)=d_{(P, \Xi_P)}(P').
\end{eqnarray*}
Thus we get 
\[
\frac{d_{(P, \Xi_P)}(P')}{c^P(P')}=\frac{d_{(X, \Gamma)}\left(
\langle P', Q\rangle\right)}{c^X\left(\langle P', Q\rangle\right)}\leq 1.
\]
It follows that $(P, \Xi_P)$ is an lc Calabi-Yau hyperplane arrangement. 

\textbf{Step 3.}
By Lemma \ref{CYempt_lemma}, we have 
$\bigcap_{H'\subset\Supp \Xi_P}H'=\emptyset$.
This implies that $Q=\bigcap_{H\subset\Supp\Gamma_P}H$. 
Hence $Q=Q(P)$ is uniquely determined by $P$. In particular, we have $Q(Q(P))=P$ and 
$(Q, \Xi_Q)$ is an lc Calabi-Yau hyperplane arrangement. 

\textbf{Step 4.}
Take an arbitrary $(c'-1)$-dimensional 
lc center $P'\subset X$ of $(X, \Gamma)$ with $P'\subset P$. 
Set $Q':=Q(P')$. Let us set $\Gamma_{P'}$, $\Gamma_{Q'}$ as in Step 1. From the 
construction, we have $\Gamma_Q\leq \Gamma_{Q'}$ and $\Gamma_P\geq \Gamma_{P'}$. 
Thus we have 
\[
Q=\bigcap_{H\subset\Supp\Gamma_P}H\subset\bigcap_{H\subset\Supp\Gamma_{P'}}H
=Q'.
\]
Moreover, it is obvious that $c^P(P\cap Q')=d_{(X, \Gamma_P)}(Q')=c'$. 

\textbf{Step 5.}
Take an arbitrary $(c'-1)$-dimensional 
lc center $P'\subset P$ of $(P, \Xi_P)$. Then, by the argument in Step 2, $P'\subset X$ 
is an lc center of $(X, \Gamma)$. Set $Q':=Q(P')$ as in Step 4. Then it follows from 
Step 4 that $c^P(P\cap Q')=d_{(P, \Xi_P)}(P\cap Q')=c'$ since $\langle P\cap Q', 
Q\rangle=Q'$. Therefore, $(P, \Xi_P)$ satisfies the assumption of 
Theorem \ref{comb_thm}. 

\textbf{Step 6.}
Let $P\subset X$ be an minimal lc center of $(X, \Gamma)$. Then, by the argument 
in Step 2, $(P, \Xi_P)$ is a klt Calabi-Yau hyperplane arrangement. Moreover, 
we can inductively show that, there exist klt Calabi-Yau hyperplane arrangements 
$(Q_1, \Xi_1),\dots,(Q_s,\Xi_s)$ such that 
$(Q,\Xi_Q)=S((Q_1,\Xi_1),\dots,(Q_s,\Xi_s))$. 
Consequently, we have 
$(X, \Gamma)=S((P, \Xi_P), (Q_1,\Xi_1),\dots,(Q_s,\Xi_s))$. 
\end{proof}

\section{On uniform K-stability}\label{unif_section}

We give a sufficient condition for log Fano pairs being uniformly K-stable or 
K-semistable. 

\begin{proposition}\label{technical_prop}
Let $(X, \Delta)$ be a K-semistable log Fano pair, set $L:=-(K_X+\Delta)$, 
let $r$ be a rational number with $r\in(0,1)$, 
and let $B$ be an effective $\Q$-divisor on $X$ with $B\sim_\Q rL$. 
Set $a:=\lct(X, \Delta; B)$. Assume that $a\geq r^{-1}$. $($Then $(X, \Delta+B)$ is a 
log Fano pair.$)$
\begin{enumerate}
\renewcommand{\theenumi}{\arabic{enumi}}
\renewcommand{\labelenumi}{(\theenumi)}
\item\label{technical_prop1}
$(X, \Delta+B)$ is K-semistable. 
Moreover, if a prime divisor $F$ over $X$ satisfies that $\hat{\beta}_{(X, \Delta+B)}(F)=0$, 
then we have $A_{(X, \Delta+aB)}(F)=0$, $\hat{\beta}_{(X, \Delta)}(F)=0$ and $a=r^{-1}$. 
\item\label{technical_prop2}
If $(X, \Delta)$ is uniformly K-stable or $a>r^{-1}$, then 
$(X, \Delta+B)$ is uniformly K-stable. 
\end{enumerate}
\end{proposition}

\begin{proof}
Set $n:=\dim X$. 
By Theorem \ref{beta_thm}, there exists $\varepsilon\geq 0$ such that 
$\hat{\beta}_{(X, \Delta)}(F)\geq \varepsilon$ holds for any prime divisor $F$ over $X$. 
Moreover, if $(X, \Delta)$ is uniformly K-stable, then we can take $\varepsilon$ to 
be positive. Note that $0\leq A_{(X, \Delta+aB)}(F)$. Thus we get 
\[
A_{(X, \Delta+B)}(F)\geq (1-a^{-1})A_{(X, \Delta)}(F).
\]
(Moreover, equality holds if and only if $A_{(X, \Delta+aB)}(F)=0$.)
This implies that 
\begin{eqnarray*}
\hat{\beta}_{(X, \Delta+B)}(F) &=& 1-\frac{(1-r)^{n+1}\int_0^\infty
\vol_X(L-xF)dx}{A_{(X, \Delta+B)}(F)\cdot(1-r)^n(L^{\cdot n})}\\
&\geq&1-\frac{1-r}{1-a^{-1}}\left(1-\hat{\beta}_{(X,\Delta)}(F)\right)\\
&\geq&\frac{1-r}{1-a^{-1}}\varepsilon+\frac{r-a^{-1}}{1-a^{-1}}.
\end{eqnarray*}
Thus we get the assertion. 
\end{proof}

\begin{corollary}\label{suff_cor}
Let $(X, \Delta)$ be an $n$-dimensional log Fano hyperplane arrangement with 
$\Delta\neq 0$. Set 
\[
\Gamma:=\frac{n+1}{d_{(X, \Delta)}}\Delta.
\]
If $(X, \Gamma)$ is an lc $($resp., a klt$)$ Calabi-Yau hyperplane arrangement, then 
$(X, \Delta)$ is K-semistable $($resp., uniformly K-stable$)$. 
\end{corollary}

\begin{proof}
By Remark \ref{mainthm_rmk} \eqref{mainthm_rmk2}, $(\pr^n,0)$ is K-semistable. 
Thus the assertion immediately follows from Proposition \ref{technical_prop}.  
\end{proof}

\section{On K-polystability}\label{poly_section}

We analyze K-polystability of log Fano hyperplane arrangements. 

\begin{corollary}\label{poly_s_cor}
Let $(X, \Gamma)$ be an $n$-dimensional log Calabi-Yau hyperplane arrangement 
of class $\sP$. Take any $d\in\Q\cap(0,n+1)$ and set 
\[
\Delta:=\frac{d}{n+1}\Gamma.
\]
Then $(X, \Delta)$ is a K-polystable log Fano pair. 
\end{corollary}

\begin{proof}
We can write $(X, \Gamma)=S((\pr^{n_1}, \Xi_1),\dots,(\pr^{n_s}, \Xi_s))$ with 
$(\pr^{n_i}, \Xi_i)$ klt Calabi-Yau hyperplane arrangements. By Corollary \ref{suff_cor}, 
we may assume that $s\geq 2$. Take an arbitrary dreamy prime divisor $F$ over 
$(X, \Delta)$. Since $s\geq 2$, we have $\lct(X, 0; \Delta)=(n+1)/d$. 
By Proposition \ref{technical_prop}, we have $\hat{\beta}_{(X, \Delta)}(F)\geq 0$. 
Moreover, if $\hat{\beta}_{(X, \Delta)}(F)=0$, then we have $A_{(X, \Gamma)}(F)=0$. 
By Proposition \ref{suff_prop}, $F$ is a product-type prime divisor over $(X, \Delta)$. 
Thus $(X, \Delta)$ is K-polystable by Theorem \ref{stps_thm}. 
\end{proof}

\begin{corollary}\label{ness_cor}
Let $(X, \Delta)$ be an $n$-dimensional K-polystable log Fano pair with $\Delta\neq 0$. 
Set 
\[
\Gamma:=\frac{n+1}{d_{(X, \Delta)}}\Delta.
\]
Then $(X, \Gamma)$ is a log Calabi-Yau hyperplane arrangement of class $\sP$. 
\end{corollary}

\begin{proof}
By Proposition \ref{blowup_prop}, $(X, \Gamma)$ is an 
lc Calabi-Yau hyperplane arrangement. 
Take an arbitrary lc center $P\subset X$ of $(X, \Gamma)$. Set $c-1:=\dim P$. 
Note that $d_{(X, \Gamma)}(P)=n+1-c$. 
Let $\sigma\colon Y\to X$ be the blowup along $P$ and set $F:=\sigma^{-1}(P)$. 
Then, from the calculation in Proposition \ref{blowup_prop}, we get 
$\hat{\beta}_{(X, \Delta)}(F)=0$. Thus $F$ must be a product-type prime divisor over 
$(X, \Delta)$ since $F$ is a dreamy prime divisor over $(X, \Delta)$. 
Let us take homogeneous coordinates 
$z_0:\cdots:z_n$ of $X$ such that $P$ is defined by the equation 
$(z_c=\cdots=z_n=0)$ and 
$\ord_F$ is the quasi-monomial valuation on $X\setminus(z_0=0)$
for coordinates $(x_i:=z_i/z_0)$ with weights 
\[\begin{cases}
0 & \text{if }1\leq i\leq c-1,\\
1 & \text{if }c\leq i\leq n,
\end{cases}\] 
and $F$ corresponds to the one-parameter subgroup 
\begin{eqnarray*}
\rho\colon\G_m &\to & \Aut(X)\\
t&\mapsto&\diag(1,\dots,1,\underbrace{t^{-1},\dots,t^{-1}}_{n+1-c 
\text{ times}})
\end{eqnarray*}
(see Example \ref{aut_P_ex}).
Set $Q:=(z_0=\cdots=z_{c-1}=0)\subset X$. 
Take any component $H$ of $\Delta$ with $P\not\subset H$. Assume that 
$H$ is defined by the equation $h_0z_0+\cdots+h_nz_n=0$. 
Since $P\not\subset H$, we have $(h_0,\dots,h_{c-1})\neq (0,\dots,0)$. 
Since $F$ is a product-type prime divisor over $(X, \Delta)$, $\rho$ factors through 
$\Aut(X, \Delta)$. In particular, for all $t\in\G_m$, we have $\rho_t^*H=H$. 
Note that $\rho_t^*H\subset X$ is defined by the equation 
\[
h_0z_0+\cdots+h_{c-1}z_{c-1}+h_ct^{-1}z_c+\cdots+h_nt^{-1}z_n=0.
\]
Since $(h_0,\dots,h_{c-1})\neq (0,\dots,0)$, we have $(h_c,\dots,h_n)=(0,\dots,0)$. 
The condition is nothing but the condition $Q\subset H$. Thus we get 
$d_{(X, \Gamma)}(Q)=c$. From Theorem \ref{comb_thm}, we get the assertion. 
\end{proof}

\begin{proof}[Proof of Theorem \ref{mainthm}]
Immediately follows from Proposition \ref{blowup_prop} and Corollaries \ref{suff_cor}, 
\ref{poly_s_cor} and \ref{ness_cor}. 
\end{proof}

\begin{proof}[Proof of Corollary \ref{maincor}]
Immediately follows from Theorem \ref{mainthm}. 
\end{proof}

\section{A relation to GIT stability}\label{GIT_section}

In this section, let us consider an $n$-dimensional hyperplane arrangement 
$(X, \Delta)=(\pr^n, \sum_{i=1}^m d_i H_i)$ with $d_1,\dots,d_m\in\Q_{>0}$ 
and $\sum_{i=1}^md_i<n+1$. 
We do \emph{not} assume that $H_1,\dots,H_m$ are mutually distinct. 
Each hyperplane $H_i\subset\pr^n$ defines a point $p_i\in{\pr^n}^\vee$ in the 
dual projective space. We set $\vec{p}:=(p_1,\dots,p_m)\in({\pr^n}^\vee)^m$. 
As in \cite[\S 11]{dolgachev}, \cite{hu} and \cite[\S 2.8]{alexeev}, 
$\SL(n+1)$ acts $({\pr^n}^\vee)^m$ naturally. Moreover, some positive multiple of 
\[
L_{\vec{d}}:=\sO_{({\pr^n}^\vee)^m}(d_1,\dots,d_m)
\]
admits a unique $\SL(n+1)$-linearization. We are interested in the 
GIT stability of the point $\vec{p}\in({\pr^n}^\vee)^m$ with respects to (some 
positive multiple of) $L_{\vec{d}}$. See \cite{GIT} and \cite{dolgachev} for the basics of 
geometric invariant theory.

\begin{definition}\label{GIT_dfn}
We say that $\vec{p}\in({\pr^n}^\vee)^m$ (with respects to $L_{\vec{d}}$) is 
\begin{enumerate}
\renewcommand{\theenumi}{\arabic{enumi}}
\renewcommand{\labelenumi}{(\theenumi)}
\item\label{GIT_dfn1}
a \emph{GIT semistable} point if there exists $s\in H^0(({\pr^n}^\vee)^m, 
l L_{\vec{d}})^{\SL(n+1)}$ for some $l\in\Z_{>0}$ such that $s(\vec{p})\neq 0$ holds,
\item\label{GIT_dfn2}
a \emph{GIT polystable} point if GIT semistable and the $\SL(n+1)$-orbit of $\vec{p}$ is closed 
in the GIT semistable locus $(({\pr^n}^\vee)^m)^{\text{ss}}(L_{\vec{d}})$, 
\item\label{GIT_dfn3}
a \emph{GIT stable} point 
if GIT polystable and the isotropy subgroup $\SL(n+1)_{\vec{p}}$ of 
$\SL(n+1)$ is finite. 
\end{enumerate}
\end{definition}

\begin{remark}\label{GIT_remark}
In \cite{GIT}, a GIT polystable point is called \emph{stable} and a GIT stable point is 
called \emph{properly stable}. In \cite{dolgachev}, a GIT polystable point is called 
\emph{Kempf-stable}. 
\end{remark}

Here is a direct consequence of Theorem \ref{mainthm}, \cite[\S 11]{dolgachev} and 
\cite{hu}. 

\begin{corollary}\label{GIT_cor}
\begin{enumerate}
\renewcommand{\theenumi}{\arabic{enumi}}
\renewcommand{\labelenumi}{(\theenumi)}
\item\label{GIT_cor1}
The following are equivalent: 
\begin{enumerate}
\renewcommand{\theenumii}{\roman{enumii}}
\renewcommand{\labelenumii}{(\theenumii)}
\item\label{GIT_cor11}
$(X, \Delta)$ is a K-semistable log Fano pair, 
\item\label{GIT_cor12}
$\vec{p}\in({\pr^n}^\vee)^m$ is a GIT semistable point with respects to $L_{\vec{d}}$. 
\end{enumerate}
\item\label{GIT_cor2}
The following are equivalent: 
\begin{enumerate}
\renewcommand{\theenumii}{\roman{enumii}}
\renewcommand{\labelenumii}{(\theenumii)}
\item\label{GIT_cor21}
$(X, \Delta)$ is a K-polystable log Fano pair, 
\item\label{GIT_cor22}
$\vec{p}\in({\pr^n}^\vee)^m$ is a GIT polystable point with respects to $L_{\vec{d}}$. 
\end{enumerate}
\item\label{GIT_cor3}
The following are equivalent: 
\begin{enumerate}
\renewcommand{\theenumii}{\roman{enumii}}
\renewcommand{\labelenumii}{(\theenumii)}
\item\label{GIT_cor31}
$(X, \Delta)$ is a K-stable log Fano pair, 
\item\label{GIT_cor32}
$(X, \Delta)$ is a uniformly K-stable log Fano pair, 
\item\label{GIT_cor33}
$\vec{p}\in({\pr^n}^\vee)^m$ is a GIT stable point with respects to $L_{\vec{d}}$. 
\end{enumerate}
\end{enumerate}
\end{corollary}

\begin{proof}
For \eqref{GIT_cor1} and \eqref{GIT_cor3}, see Theorem \ref{mainthm} and 
\cite[Theorem 11.1]{dolgachev}. For \eqref{GIT_cor2}, see Theorem \ref{mainthm} and 
\cite[Definition 3.2 and Proposition 3.3]{hu}. 
\end{proof}

Therefore, the categorical quotient 
$(({\pr^n}^\vee)^m)^{\text{ss}}(L_{\vec{d}})\sslash \SL(n+1)$
(see \cite[Theorem 8.1]{dolgachev}) parametrizes the $\SL(n+1)$-orbits of 
K-polystable log Fano hyperplane arrangements 
$(\pr^n, \sum_{i=1}^m d_iH_i)$. Note that the space 
$(({\pr^n}^\vee)^m)^{\text{ss}}(L_{\vec{d}})\sslash \SL(n+1)$ 
is a projective variety (see \cite[Proposition 8.1]{dolgachev}). 

\section{Appendix: A remark on the alpha invariants}\label{alpha_section}

Let $(X, \Delta)$ be an $n$-dimensional log Fano pair and set $L:=-(K_X+\Delta)$. 
The \emph{alpha invariant} $\alpha(X, \Delta)$  of $(X, \Delta)$ (see \cite{tian_a}) 
is defined as follows (see \cite{Dem08}): 
\[
\alpha(X, \Delta):=\inf_{D_{\geq 0}\sim_\Q L}\lct(X, \Delta; D). 
\]
Recall the following theorem (see \cite[Theorem 2.1]{tian_a}, \cite[Theorem 1.4]{OS12}, 
\cite[Theorem 1.2]{dervan}, \cite[Corollary D]{BHJ}, \cite[Theorem 3.2]{FO} and 
\cite[Theorems A and B]{BJ}): 

\begin{thm}[{\cite{tian_a, OS12, dervan, BHJ, FO, BJ}}]\label{alpha1_thm}
Assume that $\alpha(X, \Delta)\geq n/(n+1)$ $($resp., $>n/(n+1))$. Then $(X, \Delta)$ is 
K-semistable $($resp., uniformly K-stable$)$. 
\end{thm}

\begin{proof}
We give a proof for the readers' convenience. For any prime divisor $F$ over $X$, set 
\[
\tau(F):=\sup\{x\in\R_{>0}\,|\,\vol_X(L-xF)>0\}.
\]
It is easy to show the inequality $A_{(X, \Delta)}(F)\geq\tau(F)\alpha(X, \Delta)$ 
(see \cite[Lemma 3.3]{FO}). Together with \cite[Proposition 2.1]{fjt17}, we have 
\[
\hat{\beta}_{(X, \Delta)}(F)\geq 1-\frac{n}{n+1}\cdot\frac{1}{\alpha(X, \Delta)}. 
\]
Thus the assertion follows from Theorem \ref{beta_thm}. 
\end{proof}

Moreover, if $X$ is smooth, $\Delta=0$ and $n\geq 2$, we can show the following: 

\begin{thm}[{\cite[Theorem 1.2]{fjtb}}]\label{alpha_thm}
If $X$ is smooth, $\Delta=0$, $n\geq 2$ and $\alpha(X, \Delta)=n/(n+1)$, then 
$(X, \Delta)$ is K-stable. 
\end{thm}

It is natural to ask whether Theorem \ref{alpha_thm} is true or not without the assumption 
$X$ is smooth and $\Delta=0$. The following example and proposition show that it is not 
true in general. 

\begin{example}\label{alpha_ex}
Take arbitrary $n$, $m\in\Z_{>0}$ with $m\geq n+1$, and take any 
\[
t\in\biggl[\frac{m(n-1)}{(m-1)n}, 1\biggr)\cap\Q_{>0}.
\]
Set $X:=S(\pr^0, \pr^{n-1})(=\pr^n)$, let $p\in X$ be the point corresponds with $\pr^0$ 
(see Definition \ref{projective_dfn} \eqref{projective_dfn1}). 
If $n=1$, then set $\Gamma_i:=p$ for any $1\leq i\leq m$; 
if $n\geq 2$, then fix a hyperplane arrangement $(\pr^{n-1}, \sum_{i=1}^{mn}\Xi_i)$ 
with $\Xi_1,\dots,\Xi_{mn}$ general, and 
set $\Gamma_i:=\langle p, \Xi_i\rangle$ for any $1\leq i\leq mn$ (see Definition 
\ref{projective_dfn} \eqref{projective_dfn3}). 
Moreover, let $H_1,\dots,H_m\subset X$ be general hyperplanes and set 
\[
\Gamma:=\sum_{i=1}^{mn}\frac{1}{m}\Gamma_i+\sum_{j=1}^m \frac{1}{m}H_j, \quad\quad
\Delta:=t\Gamma.
\]
From the construction, $\Supp\Gamma|_{X\setminus\{p\}}$ is simple normal 
crossing. Moreover, $(X, \Gamma)$ is an lc Calabi-Yau hyperplane arrangement and is not 
of class $\sP$ (see Theorem \ref{mustata_thm} and Definition \ref{P_dfn}). 
\end{example}

\begin{proposition}\label{alpha_prop}
The above $n$-dimensional log Fano hyperplane arrangement 
$(X, \Delta)$ in Example \ref{alpha_ex} 
is K-semistable but not K-polystable. $($In particular, $(X, \Delta)$ is not K-stable.$)$ 
Moreover, we have the equality 
$\alpha(X, \Delta)=n/(n+1)$. 
\end{proposition}

\begin{proof}
By Theorem \ref{mainthm}, the log Fano hyperplane arrangement 
$(X, \Delta)$ is K-semistable but not 
K-polystable. Note that the degree of $-(K_X+\Delta)$ is equal to $(n+1)(1-t)$. 
Since the pair $(X, \Delta+n(1-t)\Gamma_1)$ satisfies that the multiplicity 
$\mult_p(\Delta+n(1-t)\Gamma_1)$ of $\Delta+n(1-t)\Gamma_1$ at $p\in X$ 
is equal to $n$, we have $\alpha(X, \Delta)\leq n/(n+1)$. 

Assume that $\alpha(X, \Delta)<n/(n+1)$. Then we may assume that $n\geq 2$. 
There exist $\varepsilon\in\Q_{>0}$, 
an effective $\Q$-divisor $D$ on $X$ of degree $n(1-t)-\varepsilon$ and a prime 
divisor $F$ over $X$ such that $(X, \Delta+D)$ is lc and $A_{(X, \Delta+D)}(F)=0$. 
For any $q\in X\setminus\{p\}$, we have 
\[
\mult_q(\Delta+D)\leq \frac{nt}{m}+n(1-t)-\varepsilon\leq 1-\varepsilon<1.
\]
Thus $(X\setminus\{p\}, (\Delta+D)|_{X\setminus\{p\}})$ is klt by 
\cite[Proposition 9.5.13]{L2}. This implies that the image of $F$ on $X$ 
must be equal to $\{p\}$. Let $\sigma\colon Y\to X$ be the blowup along $p$ and let 
$E(\simeq\pr^{n-1})\subset Y$ be the exceptional divisor of $\sigma$. Since 
\[
A_{(X, \Delta+D)}(E)=n-\mult_p(\Delta+D)\geq n-nt-(n(1-t)-\varepsilon)=\varepsilon>0, 
\]
we have $E\neq F$. In other words, $F$ is exceptional over $Y$. 
Moreover, the pair $(Y, \sigma^{-1}_*\Delta+\sigma^{-1}_*D+E)$ 
satisfies that $A_{(Y, \sigma^{-1}_*\Delta+\sigma^{-1}_*D+E)}(F)\leq 0$. 
Thus the pair $(E, \sigma^{-1}_*\Delta|_E+\sigma^{-1}_*D|_E)$ is not klt by inversion 
of adjunction (see \cite[Theorem 5.50]{KoMo}). Note that $\sigma^{-1}_*\Delta|_E$ is 
equal to $(t/m)\sum_{i=1}^{mn}\Xi_i$ under the natural identification $E=\pr^{n-1}$, 
and the degree of $\sigma^{-1}_*D|_E$ is at most $n(1-t)-\varepsilon$. Thus, for any 
$p'\in E$, we have 
\[
\mult_{p'}(\sigma^{-1}_*\Delta|_E+\sigma^{-1}_*D|_E)\leq
\frac{(n-1)t}{m}+n(1-t)-\varepsilon<1.
\]
Again by \cite[Proposition 9.5.13]{L2}, the pair 
$(E, \sigma^{-1}_*\Delta|_E+\sigma^{-1}_*D|_E)$ must be klt. 
This leads to a contradiction. Thus we have $\alpha(X, \Delta)=n/(n+1)$. 
\end{proof}

\end{document}